\documentclass[11pt,a4paper]{amsart}
\usepackage[margin=20mm]{geometry}
\usepackage[numbers]{natbib}
\usepackage{hyperref}

\usepackage{amsaddr}

\usepackage{amssymb,amsmath}
\usepackage{amsthm}
\usepackage{bbm}
\usepackage{stmaryrd}
\usepackage[usenames,dvipsnames]{xcolor}
\usepackage{color}
\usepackage{enumitem}

\input xy
\xyoption{all} %\tolerance=500

\numberwithin{equation}{section}

\newtheorem{theorem}{Theorem}[section]
\newtheorem{corollary}{Corollary}[section]
\newtheorem*{main}{Main Theorem}
\newtheorem{lemma}{Lemma}[section]
\newtheorem{proposition}{Proposition}[section]

\theoremstyle{definition}
\newtheorem{definition}{Definition}[section]
\newtheorem{remark}{Remark}[section]
\newtheorem{example}{Example}[section]

\newtheorem*{motivation}{Motivating Example}

\newcommand{\rmd}{\textnormal{d}}

\DeclareMathOperator{\Vect}{Vect}

\DeclareMathOperator{\Span}{Span}

\newcommand{\catname}[1]{\textnormal{\texttt{#1}}}

\font\black=cmbx10 \font\sblack=cmbx7 \font\ssblack=cmbx5 \font\blackital=cmmib10  \skewchar\blackital='177
\font\sblackital=cmmib7 \skewchar\sblackital='177 \font\ssblackital=cmmib5 \skewchar\ssblackital='177
\font\sanss=cmss10 \font\ssanss=cmss8 %scaled 900
\font\sssanss=cmss8 scaled 600 \font\blackboard=msbm10 \font\sblackboard=msbm7 \font\ssblackboard=msbm5
\font\caligr=eusm10 \font\scaligr=eusm7 \font\sscaligr=eusm5  \font\fraktur=eufm10
\font\sfraktur=eufm7 \font\ssfraktur=eufm5 
\font\bsymb=cmsy10 scaled\magstep2
\def\all#1{\setbox0=\hbox{\lower1.5pt\hbox{\bsymb
       \char"38}}\setbox1=\hbox{$_{#1}$} \box0\lower2pt\box1\;}
\def\exi#1{\setbox0=\hbox{\lower1.5pt\hbox{\bsymb \char"39}}
       \setbox1=\hbox{$_{#1}$} \box0\lower2pt\box1\;}

\def\tx#1{{\fam0\relax#1}}

\newfam\bifam
\textfont\bifam=\blackital \scriptfont\bifam=\sblackital \scriptscriptfont\bifam=\ssblackital

\newfam\blfam
\textfont\blfam=\black \scriptfont\blfam=\sblack \scriptscriptfont\blfam=\ssblack

\newfam\bbfam
\textfont\bbfam=\blackboard \scriptfont\bbfam=\sblackboard \scriptscriptfont\bbfam=\ssblackboard

\newfam\ssfam
\textfont\ssfam=\sanss \scriptfont\ssfam=\ssanss \scriptscriptfont\ssfam=\sssanss
\def\sss#1{{\fam\ssfam\relax#1}}

\newfam\clfam
\textfont\clfam=\caligr \scriptfont\clfam=\scaligr \scriptscriptfont\clfam=\sscaligr

\newfam\frfam
\textfont\frfam=\fraktur \scriptfont\frfam=\sfraktur \scriptscriptfont\frfam=\ssfraktur

\def\hpb#1{\setbox0=\hbox{${#1}$}
    \copy0 \kern-\wd0 \kern.2pt \box0}
\def\vpb#1{\setbox0=\hbox{${#1}$}
    \copy0 \kern-\wd0 \raise.08pt \box0}

\def\pmb#1{\setbox0\hbox{${#1}$} \copy0 \kern-\wd0 \kern.2pt \box0}
\def\pmbb#1{\setbox0\hbox{${#1}$} \copy0 \kern-\wd0
      \kern.2pt \copy0 \kern-\wd0 \kern.2pt \box0}
\def\pmbbb#1{\setbox0\hbox{${#1}$} \copy0 \kern-\wd0
      \kern.2pt \copy0 \kern-\wd0 \kern.2pt
    \copy0 \kern-\wd0 \kern.2pt \box0}
\def\pmxb#1{\setbox0\hbox{${#1}$} \copy0 \kern-\wd0
      \kern.2pt \copy0 \kern-\wd0 \kern.2pt
      \copy0 \kern-\wd0 \kern.2pt \copy0 \kern-\wd0 \kern.2pt \box0}
\def\pmxbb#1{\setbox0\hbox{${#1}$} \copy0 \kern-\wd0 \kern.2pt
      \copy0 \kern-\wd0 \kern.2pt
      \copy0 \kern-\wd0 \kern.2pt \copy0 \kern-\wd0 \kern.2pt
      \copy0 \kern-\wd0 \kern.2pt \box0}

\mathchardef\za="710B  %\alpha
\mathchardef\zb="710C  %\beta
\mathchardef\zg="710D  %\gamma
\mathchardef\zd="710E  %\delta
\mathchardef\zve="710F %\epsilon
\mathchardef\zz="7110  %\zeta
\mathchardef\zh="7111  %\eta
\mathchardef\zvy="7112 %\theta
\mathchardef\zi="7113  %\iota
\mathchardef\zk="7114  %\kappa
\mathchardef\zl="7115  %\lambda
\mathchardef\zm="7116  %\mu
\mathchardef\zn="7117  %\nu
\mathchardef\zx="7118  %\xi
\mathchardef\zp="7119  %\pi
\mathchardef\zr="711A  %\rho
\mathchardef\zs="711B  %\sigma
\mathchardef\zt="711C  %\tau
\mathchardef\zu="711D  %\upsilon
\mathchardef\zvf="711E %\phi
\mathchardef\zq="711F  %\chi
\mathchardef\zc="7120  %\psi
\mathchardef\zw="7121  %\omega
\mathchardef\ze="7122  %\varepsilon
\mathchardef\zy="7123  %\vartheta
\mathchardef\zf="7124  %\varomega
\mathchardef\zvr="7125 %\varrho
\mathchardef\zvs="7126 %\varsigma
\mathchardef\zf="7127  %\varphi
\mathchardef\zG="7000  %\Gamma
\mathchardef\zD="7001  %\Delta
\mathchardef\zY="7002  %\Theta
\mathchardef\zL="7003  %\Lambda
\mathchardef\zX="7004  %\Xi
\mathchardef\zP="7005  %\Pi
\mathchardef\zS="7006  %\Sigma
\mathchardef\zU="7007  %\Upsilon
\mathchardef\zF="7008  %\Phi
\mathchardef\zW="700A  %\Omega
\mathchardef\zC="7009  %\Psi

\newcommand{\be}{\begin{equation}}
\newcommand{\ee}{\end{equation}}

\newcommand{\bea}{\begin{eqnarray}}
\newcommand{\eea}{\end{eqnarray}}
\def\*{{\textstyle *}}
\newcommand{\R}{{\mathbb R}}

\newcommand{\Z}{{\mathbb Z}}

\newcommand{\s}{{\textstyle *}}

\def\ul{\underline}
\def\Sec{\sss{Sec}}
\def\Vect{\sss{Vect}}

\def\sT{{\sss T}}

\def\xi{\tx{i}}

\def\s*{{\scriptstyle *}}

\def\cO{\mathcal{O}}

\def\ul{\underline}
\newcommand{\beas}{\begin{eqnarray*}}
\newcommand{\eeas}{\end{eqnarray*}}
\title{On a Grassmann odd analogue of Carrollian Manifolds}
\author{Andrew James Bruce}  
   \email{andrewjamesbruce@googlemail.com}
  
   \date{\today}
 
\begin{document}
\begin{abstract}
 We define a Grassmann odd analogue of a Carrollian manifold as a supermanifold of dimension $n|1$ with an even degenerate metric such that the kernel is generated by a non-singular odd vector field that is a supersymmetry generator. Alongside other results, we establish that the reduced manifold is a pseudo-Riemannian manifold, and show that compatible affine connections always exist, albeit they must carry torsion. As a physically relevant example, we examine an Inönü--Wigner contraction of the supertranslation algebra on standard superspace $\mathbb{R}^{4|4}$.    \par
\end{abstract}

 \maketitle

\setcounter{tocdepth}{2}
 \tableofcontents
 
\begin{flushright}
\emph{“Curiouser and curiouser!” cried Alice }\\
Lewis Carroll, Alice’s Adventures in Wonderland,  (1865) 
\end{flushright}

\section{Introduction}
Carrollian manifolds are understood as manifolds equipped with a degenerate metric whose kernel is spanned by a nowhere vanishing complete vector field (see Duval et al. \cite{Duval:2014,Duval:2014a,Duval:2014b}, and the earlier works of Lévy-Leblond \cite{Lévy-Leblond:1965},  Sen Gupta \cite{SenGupta:1966}, and  Henneaux \cite{Henneaux:1979}). Carrollian geometry/physics has grown from a mathematical curiosity based on the non-relativistic limit $c \rightarrow 0$ to a subject of ongoing research. For a review of many of the facets of Carrollian physics, the reader may consult Bagchi et al. \cite{Bagchi:2025}. Intrinsic approaches to Carrollian geometry, so working with geometries not directly associated with a limit, have been developed, starting with the work of Duval et al. \cite{Duval:2014,Duval:2014a,Duval:2014b} (also see Henneaux \cite{Henneaux:1979}). Null hypersurfaces, such as punctured future or past light-cones in Minkowski spacetime, and the event horizon of a Schwarzschild black hole, are examples of Carrollian manifolds.\par
Supermanifolds offer the possibility of generalisations of classical geometries that are described by odd structures.  For example, the notion of odd connections on supermanifolds was studied by the author \& Grabowski \cite{Bruce:2020a}; Khudaverdian \& Peddie \cite{Khudaverdian:2016} provided a comparison between odd Riemannian and odd symplectic; and Khudaverdian \& Voronov \cite{Khudaverdian:2002} examined odd Laplace operators. Thus, as a mathematical question, the possibility of Grassmann odd analogues of Carrollian geometries is raised. Specifically, one can consider an odd vector field that generates the kernel of a degenerate metric. Shander \cite{Shander:1980} provides the local form of the (non-singular) odd vector fields\
\def\arraystretch{1.5}
\begin{tabular}{lll}
\textbf{Homological:}     &  $Q =  \frac{\partial}{\partial \tau}$, &$Q^2 = \mathbf{0}$, \\
\textbf{Supersymmetric:}  &  $Q = \frac{\partial}{\partial \tau} + \tau \frac{\partial}{\partial t}$,& $Q^2 = \frac{\partial}{\partial t}$,
\end{tabular}\par
\noindent where $t$ is an even coordinate and $\tau$ an odd coordinate. It is the supersymmetric generalisation that we study here, as the non-integrable distribution is quite at odds with the classical situation. Moreover, Vaintrob \cite{Vaintrob:1996} has shown that if a supermanifold admits a non-singular homological vector field, then the supermanifold is a trivial odd line bundle. We will restrict our attention to even degenerate metrics, although odd metrics can be studied; they seem less relevant in physics\footnote{The author is not aware of any direct application of odd Riemannian metrics in physics. This is in contrast to the use of odd symplectic structures in the BV formalism. See Khudaverdian \cite{Khudaverdian:2004} and references therein.}.\par
The core concept introduced in the paper is that of a \emph{super-null Riemannian manifold}. That is, a supermanifold with a single odd direction that is equipped with a degenerate (even) Riemannian metric $g$, such that the kernel of the metric is generated by a non-singular odd vector field $Q$, that satisfies $Q^2 \neq \mathbf{0}$; see Definition \ref{def:SUSYCarBun}. In particular,  $[Q, Q] = 2 P$, and $[P,Q]=0$, and so we have the $N = 1$, $d =1$ supertranslation algebra describing the kernel of the degenerate Riemannian metric.  We avoid the nomenclature ``super-Carrollian manifolds’’ or similar to avoid confusion with supersymmetric extensions of Carrollian geometry and/or the Carroll group.  We stress that the structures we study are not associated with Carrollian superalgebras as studied in the context of Carrollian supergravities, see for example, Bergshoeff, Gomis, \& Parra \cite{ Bergshoeff:2016},  Ravera \& Zorbra \cite{Ravera:2023}, and Grumiller, Montecchio, \& Nejati \cite{Grumiller:2024}, and references therein.  \par
We define the \emph{super-null Lie algebra} as the Lie superalgebra of infinitesimal automorphisms of a super-null Riemannian manifold, i.e.,  vector fields $X$ that satisfy $\mathcal{L}_X g =0$ and $\mathcal{L}_XQ =0$. One observation is that this Lie superalgebra is pure even if $P$ is not Killing, and if $P$ is Killing, the odd elements are spanned by the supersymmetric covariant derivative. Affine connections on a super-null Riemannian manifold are carefully studied. Supersymmetric and metric compatible affine connections are defined and examined; the former are connections for which the vector field $Q$ is parallel, and the latter are the natural generalisation of metric compatible connections to degenerate metrics. As the metric on a super-null Riemannian manifold is degenerate, there is no direct generalisation of the fundamental theorem of Riemannian supergeometry. We prove that both supersymmetric compatible and metric compatible connections exist on any super-null Riemannian manifold; however, they are not uniquely fixed by the metric and odd vector field. Moreover, due to the distribution given by the kernel of the metric being non-integrable, these connections must carry torsion.

\medskip
\begin{motivation}\label{exm:SUSY} Starting with flat superspace, one can construct a super-null Riemannian manifold using an Inönü–Wigner contraction (see \cite{Inönü:1953}). Detailed definitions and presentation of conventions can be found throughout Subsections \ref{subsec:MetricsAndConnections}  and \ref{subsec:SUSYBundles}. We will follow the conventions of \cite{Bruce:2024} for spinors and gamma matrices. Consider $N=1$, $d=4$ flat superspace $\R^{4|4}$, with global coordinates $(x^\mu,\theta_\alpha)$, where the odd coordinates are real Majorana spinors. The supersymmetry generators are
\begin{equation}\label{eqn:D4SUSY}
Q^\alpha = \frac{\partial}{\partial \theta_\alpha} + \theta_\beta (C\gamma^\mu)^{\beta \alpha} \frac{\partial}{\partial x^\mu}\,.
\end{equation}
The (graded) commutator is
$$[Q^\alpha,Q^\beta] = 2\,(C\gamma^\mu)^{\beta\alpha} P_\mu,
\qquad P_\mu = \frac{\partial}{\partial x^\mu}\,.$$
To take an Inönü–Wigner contraction, we separate spatial and temporal directions and write
$$
(x^\mu,\theta_\alpha) = (x^i,t,\theta_a,\tau)\,.
$$
We then rescale the spatial coordinates
\begin{equation}
x^i \longmapsto \hat{x}^i := c^{-1} x^i\,,
\qquad \theta_a \longmapsto \hat{\theta}_a := \sqrt{c}\,\theta_a\,,
\end{equation}
where $c$ is the speed of light, regarded as a contraction parameter. The temporal coordinates $(t,\tau)$ are not rescaled. Furthermore, we rescale the generators
\begin{equation}
Q^a \longmapsto \hat{Q}^a := \sqrt{c} \,Q^a\,, \qquad P_i \longmapsto \hat{P}_i := c\, P_i\,,
\end{equation}
to ensure that the (graded) commutators remain finite in the limit $c\to 0$; the rescaled generators vanish except for the surviving pair
$$Q := Q^\tau = \frac{\partial}{\partial \tau} + \tau \frac{\partial}{\partial t}\,,  \qquad P := P_t = \frac{\partial}{\partial t}\,.$$
One finds that the non-trivial bracket is
$$[Q,Q] = 2\, P\,,$$
which is precisely the $d=1$, $N=1$ supertranslation algebra. After the contraction, we interpret these vector fields as being vector fields on $\R^{4|1}\subset \R^{4|4}$ (using an abuse of notation), which is defined by setting $\theta_a =0$, and so comes with global coordinates $(x^i, t , \tau)$.  A canonical choice of degenerate metric (not unique) whose kernel is generated by $Q$ is
$$g   = \rmd x^i \otimes \rmd x^j \, \delta{ji}  + 2 \, \rmd t \otimes \rmd \tau \, \tau - \rmd t \otimes \rmd t\,.$$
Here we have employed the convention that $\rmd \tau$ is odd, meaning that the degenerate metric is even. Using the linearity of the associated pseudo-inner product, we observe that
\begin{align*}
& \langle Q|  \partial_i\rangle = \langle \partial\tau | \partial_i\rangle + \tau\, \langle \partial_t | \partial_i\rangle =0\,,\\
&\langle Q|  \partial_t\rangle = \langle \partial_\tau | \partial_t\rangle + \tau\, \langle \partial_t | \partial_t\rangle =\tau -\tau=0\,,\\
& \langle Q|  \partial_\tau\rangle = \langle \partial_\tau | \partial_\tau\rangle + \tau\, \langle \partial_t | \partial_\tau\rangle =\tau^2=0\,,
\end{align*}
which demonstrates that the kernel of $g$ is generated by $Q$. Notice that the reduced metric is the Minkowski metric on $\R^4$.
\begin{remark}
It must be stressed that the Inönü–Wigner contraction employed above does \ul{not} yield a Carrollian structure as the underlying Poncaré group is not contracted.  That is, we do not have a super-generalisation of a Carrollian structure, but rather an odd analogue thereof in that the fundamental vector field is now odd.
\end{remark}
The above constructions will form the basis of examples \ref{exa:SUSYLie} and \ref{exa:Connection}.
\end{motivation}
\par
The general definition of a super-null Riemannian manifold (see Definition \ref{def:SUSYCarBun})  covers the above example; however, there will be no assumption that the supermanifold necessarily arises from an Inönü–Wigner contraction. That is, we take an intrinsic point of view. \par
\medskip
\noindent \textbf{Key Results.}
\begin{itemize}[itemsep=0.5em]
\item \hyperref[mtrm:ComConExist]{Main Theorem} - affine connections that are both supersymmetric compatible and metric compatible exist on any super-null Riemannian manifold, and they carry non-vanishing torsion;
\item Proposition \ref{prop:LocalFormg} - the local form of the degenerate metric in Shander coordinates is presented;
\item Proposition \ref{prop:RedMisRie} - the reduced metric is non-degenerate, that is, underlying a super-null Riemannian manifold is a pseudo-Riemannian manifold;
\item Proposition \ref{prop:DNotKilling} - the odd vector field Q cannot be Killing;
\item Proposition \ref{prop:SCarrAlg} + Corollary \ref{cor:SCarrAlfFinite} - the  super-null Lie algebra is finite dimensional.
\end{itemize}
\medskip
\noindent \textbf{Arrangement.} In Section \ref{sec:Main} we proceed with the bulk of this paper. We recall the fundamental theory of supermanifolds equipped with (degenerate) metrics and connections in Subsection \ref{subsec:MetricsAndConnections}.  In Subsection \ref{subsec:SUSYBundles}, the concept of a super-null Riemannian manifold is presented, and some immediate results are given. The super-null Lie algebra is studied in Subsection \ref{subsec:SupCarLieAlg}. The question of compatible affine connections is addressed in Subsection \ref{subsec:ConnSUSYBun}. We end in Section \ref{sec:ConRem} with some concluding remarks.
%
%%%%%%%%%%%%%%%%%%%%%%%
%
\section{Degenerate Metrics, Connections and Supersymmetric Structures}\label{sec:Main}
\subsection{Even Metrics and Connections on Supermanifolds}\label{subsec:MetricsAndConnections}
We will assume that the reader has some familiarity with the category of (real and finite-dimensional) supermanifolds $\catname{SMan}$.  We understand a \emph{supermanifold} $M := (|M|, :  \cO_{M})$ of dimension $n|m$  to be a supermanifold as defined by Berezin \& Leites  \cite{Berezin:1976,Leites:1980}, i.e., we take the locally ringed space approach. For an overview of the general theory of supermanifolds, the reader may consult, for example, \cite{Carmeli:2011,Manin:1997,Varadrajan:2004}. Underlying any supermanifold is a smooth manifold that we will denote $M_{red} = (|M|, C^\infty_{|M|}(-))$. An incomplete list of works on Riemannian supermanifolds includes \cite{Galaev:2012,Garnier:2014,Garnier:2012,Goertsches:2008,Groeger:2014,Klaus:2019}. The warning from the outset is that we will include degenerate metrics in our definition.
\begin{definition}
A \emph{metric} on a supermanifold $M$ is an even, rank $2$, ($\Z_2$-graded) symmetric covariant tensor $g \in \Sec\big( \sT^* M \otimes \sT^* M\big)$.
\end{definition}
In the local coordinate frame, we write
$$g = \rmd x^a \otimes \rmd x^b \,g_{ba}(x)\,,$$
where we have assigned the parity $\widetilde{\rmd x^a} = \widetilde{a}$. Thus, $\widetilde{g_{ba} } = \widetilde{a} + \widetilde{b}$. Under changes of coordinates $x^a \mapsto x^{a'}(x)$ the local frame  and components of the metric transforms as
$$\rmd x^{a'} = \rmd x^a \left(\frac{\partial x^{a'}}{\partial x^a}\right)\,, \qquad g_{b'a'}(x') = (-1)^{\widetilde{a}' \, \widetilde{b}} \left(\frac{\partial x^b}{\partial x^{b'}}\right) \left(\frac{\partial x^a}{\partial x^{a'}}\right)~  g_{ab}\,,$$
where we have explicitly used the symmetry $g_{ab} = (-1)^{\widetilde{a} \, \widetilde{b}} \, g_{ba}$. The \emph{(semi-)inner product} associated with $g$ is locally given by
\begin{equation}
\langle X| Y \rangle =  (-1)^{\widetilde{Y}\, \widetilde{a}} ~ X^a(x)Y^b(x) g_{ba}(x)\,.
\end{equation}
We have the following properties that can be checked directly:
\begin{enumerate}[itemsep=0.5em]
\item$\widetilde{\langle X| Y \rangle} = \widetilde{X} + \widetilde{Y}$;
\item $\langle X| Y \rangle = (-1)^{\widetilde{X} \, \widetilde{Y}} \langle Y| X \rangle$;
\item $\langle f X + Y| Z \rangle = f \langle X|Z \rangle + \langle Y| Z\rangle$;
\item $\langle \partial_a | \partial_b \rangle = g_{ab}$,
\end{enumerate}
\medskip
for all (homogeneous) $X,Y,Z \in \Vect(M)$ and $f \in C^\infty(M)$. Extension of these properties to inhomogeneous vector fields is by linearity.  A vector field $X \in \Vect(M)$ is said to be a \emph{null vector field} if  $\langle X | X \rangle =0$. Note that odd vector fields are automatically null vector fields as, via the symmetry of the inner product  $\langle X | X \rangle =  -\, \langle X | X \rangle =0$.
\begin{definition}
Let $g$ be a metric on a supermanifold $M$. The \emph{kernel of the metric} $g$ is the $C^\infty(M)$-module
$$\ker(g) :=  \big \{ X \in \Vect(M)~~|~~ \langle X |Y\rangle   =0\,, ~\textnormal{for all} ~~Y \in \Vect(M) \big \}\,.$$
A metric $g$ is said to be \emph{non-degenerate} if $\ker(g) = {\mathbf{0}}$, and is said to be \emph{degenerate} otherwise.
\end{definition}
\begin{definition}
A pair $(M,g)$, where $M$ is a supermanifold and $g$ is a metric is said to be
\begin{enumerate}[itemsep=0.5em]
\item a \emph{Riemannian supermanifold} if $g$ is non-degenerate; and
\item a \emph{degenerate Riemannian supermanifold} is $g$ is degenerate.
\end{enumerate}
\end{definition}
\begin{remark}
Odd metrics can also be similarly defined; however, we will not discuss them in this paper.  It is well known that a Riemannian supermanifold (with an even metric) must have dimensions $n|2m$.
\end{remark}
Killing vector fields are defined in the same way as in classical Riemannian geometry.
\begin{definition}\label{def:KillVF}
Let $(M, g)$ be a (degenerate) Riemannian supermanifold. A vector field $X \in \Vect(M)$ is said to be a Killing vector field if
$$\mathcal{L}_X g =0\,.$$
\end{definition}
A useful expression for the Lie derivative of the metric is
\begin{equation}\label{eqnm:LieDerMet}
(\mathcal{L}_X g)(Y,Z) = X \langle Y|Z\rangle - \langle [X,Y]| Z\rangle - (-1)^{\widetilde{X} \widetilde{Y}}\,  \langle Y| [X,Z]\rangle \,,
\end{equation}
for all $X,Y, Z \in \Vect(M)$. Naturally, this local expression is identical to the classical one up to some sign factors.
\begin{proposition}\label{prop:KillingLieBracket}
The set of all Killing vector fields on (degenerate) Riemannian supermanifold $(M,g)$ forms a Lie algebra with respect to the standard Lie bracket of vector fields on $M$.
\end{proposition}
\begin{proof}
This follows in complete parallel with the classical case using $\mathcal{L}_{[X,Y]} = [\mathcal{L}_X, \mathcal{L}_Y]$.
\end{proof}
The notion of an affine connection on a supermanifold is more or less the same as that of an affine connection on a manifold.
\begin{definition}\label{def:AffCon}
An \emph{affine connection} on a supermanifold is a parity-preserving map
\begin{align*}
\nabla : &  ~ \Vect(M) \times \Vect(M) \longrightarrow \Vect(M)\\
& (X,Y) \mapsto \nabla_X Y
\end{align*}
that satisfies the following:
\begin{itemize}
\item Bi-linearity
\begin{align*}
& \nabla_X(Y+Z) = \nabla_X Y + \nabla_X Z\,,
&&\nabla_{X +Y} Z = \nabla_X Z + \nabla_Y Z\,,
\end{align*}
\item $C^\infty(M)$-linearity in the first argument
$$\nabla_{fX} Y = f \, \nabla_X Y\,,$$
\item The Leibniz rule
$$\nabla_X fY = X(f) Y + (-1)^{\widetilde{X} \, \widetilde{f}} ~ f \, \nabla_X Y\,,$$
\end{itemize}
for all  (homogeneous) $X,Y,Z \in \Vect(M)$ and $f \in C^\infty(M)$. Extension to inhomogeneous vector fields is by linearity.
\end{definition}
Affine connections exist on any (real) supermanifold. This can be proved by adapting the standard arguments, i.e., affine connections are local operators and the existence of a partition of unity\footnote{See \cite[Lemma 3.1.7 and Corollary 3.1.8]{Leites:1980} for the existence of partitions of unity and bump functions on supermanifolds.}.
\begin{definition}\label{def:SymCon}
Let $\nabla$ be an affine connection on a supermanifold $M$.  The \emph{torsion tensor} of an affine connection  $T_\nabla : \Vect(M)\otimes_{C^\infty(M)} \Vect(M) \rightarrow \Vect(M)$   is defined as
$$T_\nabla(X,Y) := \nabla_X Y - (-1)^{\widetilde{X} \, \widetilde{Y} } ~ \nabla_Y X - [X,Y]\,,$$
for any homogeneous $X, Y \in \Vect(M)$. An affine connection is said to be \emph{symmetric}  or \emph{torsion-free} if the torsion vanishes.
\end{definition}
\begin{definition}\label{def:RieCurTensor}
Let $\nabla$ be an affine connection on a supermanifold $M$.  The \emph{Riemann curvature tensor} of an affine connection  $R_\nabla : \Vect(M)\otimes_{C^\infty(M)} \Vect(M)\otimes_{C^\infty(M)} \Vect(M) \rightarrow \Vect(M)$   is defined as
$$R_\nabla(X,Y) := \nabla_X (\nabla_Y Z) - (-1)^{\widetilde{X}\, \widetilde{Y}}\, \nabla_Y (\nabla_X Z) - \nabla_{[X,Y]}Z\,,$$
for any homogeneous $X, Y \in \Vect(M)$ and $Z \in \Vect(M)$. An affine connection is said to be \emph{flat} if the Riemann curvature tensor vanishes.
\end{definition}
\begin{definition}\label{def:MetComp}
An affine connection $\nabla$ on a (degenerate) Riemannian supermanifold $(M,g)$ is said to be \emph{metric compatible} if
$$X\langle Y|Z \rangle  = \langle \nabla_X Y|Z \rangle  + (-1)^{\widetilde{X}\, \widetilde{Y}} ~ \langle Y|\nabla_X Z \rangle\,,$$
for any $X, Y, Z \in \Vect(M)$.
\end{definition}
\begin{remark}
The fundamental theorem of Riemannian geometry generalises directly to the case of Riemannian supermanifolds. That is, there is a unique metric compatible and torsion-free affine connection on an even Riemannian supermanifold, i.e., the Levi-Civita connection. However, there is, in general, no analogue for degenerate metrics; such connections may not exist, and if they do, they are usually not unique.
\end{remark}
%
%%%%%%%%%%%%%%%%%%%%%%%
%
\subsection{Super-Null Riemannian Manifolds}\label{subsec:SUSYBundles}
Generalising the definition of a Carrollian manifold (see Duval et al. \cite{Duval:2014,Duval:2014a,Duval:2014b}), we make the following definition.
\begin{definition}\label{def:SUSYCarBun}
A \emph{super-null Riemannian manifold} is a quadruple $(M, g, Q, P)$, where
\begin{enumerate}[itemsep=0.5em]
\item $M =(|M|, \cO_M)$ is a supermanifold of dimension $n|1$;
\item $(M, g)$ is a degenerate (even) Riemannian supermanifold;
\item $Q \in \Vect(M)$  is a non-singular odd vertical vector field  such that $[Q,Q] = 2 \,P$, where $P$ is an even vector field on $M$,
\end{enumerate}
subject to the compatibility condition $\ker(g) = \Span {Q}$.\par 
A \emph{morphism of super-null Riemannian manifolds} $\Phi :(M, g, Q, P) \rightarrow (M', g', Q', P') $ is a diffeomorphism  $\Phi : M \rightarrow M'$, such that
\begin{enumerate}[itemsep=0.5em]
\item $g = \Phi^* g'$; and
\item  $\Phi_* Q = Q'$.
\end{enumerate}
The resulting \emph{category of super-null Riemannian manifolds} is denoted $\catname{SNullMan}$. The \emph{group of automorphism of a super-null Riemannian manifold} we denote as $\mathrm{SNull}(M, g, Q, P)$, and the associated  \emph{super-null Lie algebra} $\mathfrak{snull}(M, g, Q, P)$ is the Lie superalgebra of vector fields $X \in \Vect(M)$ that satisfy
$$\mathcal{L}_X g =0\,, \qquad  \mathcal{L}_X Q =0\,.$$
\end{definition}
\noindent \textbf{Observations:}
\begin{enumerate}[itemsep=0.5em]
\item the rank of the kernel is constant, which ensures the kernel is well-defined as a locally free module;
\item the supermanifold $M$ we take to be of dimension $n|1$, with $n \geq 1$, this is minimal case and parallels $N=1$ SUSY-curves. Moreover, $M$ is an odd line bundle, i.e., locally $M$ is of the form $|U| \times \R^{0|1}$;
\item the Lie superalgebra here is $[Q,Q] = 2\,P$, $[P,Q] =0$ and $[P,P] =0$, i.e., the $N=1$, $d=1$ supertranslation algebra;
\item Shander’s theorem \cite{Shander:1980} means that in the neighbourhood of any point $m \in |M|$, adapted coordinates can always be found, which we will refer to as \emph{Shander coordinates}, $(x^a, t, \tau)$, such that  the vector fields take the canonical form
$$Q = \frac{\partial}{\partial \tau} + \tau \frac{\partial}{\partial t}\,, \qquad P = \frac{\partial}{\partial t}\,,$$
where $\widetilde{x^a} = \widetilde{t} = 0$ and $\widetilde{\tau} = 1$;
\item morphisms equate the kernels, i.e., $\Phi^* \ker(g) = \ker(g')$, and $\Phi^* P = P'$ as the push-forward by a diffeomorphism is a homomorphism of Lie brackets;
\item if $X \in \mathfrak{snull}(M, g, Q, P) $, then $\mathcal{L}_X P = [[X,Q], Q] =0$.
\end{enumerate}
\begin{small}
\noindent \textbf{Aside.} \emph{Conformal morphisms} of super-Carrollian manifolds are defined as diffeomorphisms $\Phi : M \rightarrow M'$ such that 
\begin{enumerate}[itemsep=0.5em]
\item $\Phi^* g' = \lambda^2 \, g$, where $\lambda \in C^\infty(M)$ is even and nowhere vanishing; and 
\item $\Phi_* \ker(g) = \ker(g')$. \label{eqn:ConMorKer}
\end{enumerate}
Condition (\ref{eqn:ConMorKer}) implies $\Phi_* Q = \mu \, Q$, where $\mu \in C^\infty(M)$ is even and nowhere vanishing. As the pushforward by a diffeomorphism respects the Lie bracket of vector fields, a quick calculations shows that $\Phi_* P = \big(\Phi^*(\mu)\big)^2 \, P' + \Phi^*(\mu) \, Q'\big(\Phi^* \mu \big)$.
\end{small}
\begin{proposition}\label{prop:LocalFormg}
Let $(M, g, Q, P)$ be a super-null Riemannian manifold. Then in Shander coordinates $(x^a, t , \tau)$, the most general form of the degenerate metric $g$ is
\begin{align*}
g &= \rmd x^a \otimes \rmd x^b \, g_{ba}(x,t)  + 2\, \rmd x^a \otimes \rmd t \,  g_{ta}(x,t) \\
&- 2 \, \rmd x^a \otimes \rmd \tau\, \tau \, g_{ta}(x,t) -  2\,\rmd t \otimes \rmd \tau\, \tau \, g_{tt}(x,t) + \rmd t \otimes \rmd t \, g_{tt}(x,t)\,.
\end{align*}
\end{proposition}
\begin{proof}
In Shander coordinates $(x^a, t, \tau)$, the most general form of a (possibly degenerate) metric on $M$ is of the form
\begin{align*}
g =& \rmd x^a \otimes \rmd x^b \, g_{ba}(x,t)  + 2\, \rmd x^a \otimes \rmd t \,  g_{ta}(x,t) \\
&+ 2\, \rmd x^a \otimes \rmd \tau \,  \tau \, h_{\tau a}(x,t)  +  2\, \rmd t \otimes \rmd \tau \,  \tau \, h_{\tau t}(x,t) + \rmd t \otimes \rmd t \,  g_{tt}(x,t)  \,.
\end{align*}
Using $C^\infty(M)$-linearity, we need only examine the following
\begin{itemize}[itemsep=0.5em]
\item $\langle Q | \partial_a \rangle = \langle \partial_\tau | \partial_a \rangle + \tau \langle \partial_t | \partial_a \rangle =   \tau \, h_{\tau a} +  \tau \,g_{t a}$;
\item $\langle Q | \partial_t \rangle = \langle \partial_\tau | \partial_t \rangle + \tau \langle \partial_t | \partial_t \rangle =  \tau \, h_{\tau t} +  \tau \, g_{t t}$;
\item $\langle Q | \partial_\tau \rangle = \langle \partial_\tau | \partial_\tau \rangle + \tau \langle \partial_t | \partial_\tau \rangle =   \tau (\tau \, h_{\tau t}) =0$.
\end{itemize}
For the above to vanish, we require
$$h_{\tau a} = - g_{ta}\,, \qquad  h_{\tau t } = - g_{t t}\,,$$
and then substituting this into the general possible form of $g$ establishes the result.
\end{proof}
\begin{remark}
Locally, $\tau Q = \tau \partial_\tau$, and so these vector fields are not linearly independent. Thus,  $\langle Q | \partial_\tau \rangle =0$ is consistent with the rank of $\ker(g)$ being generated by $Q$. Moreover, we have not yet placed constraints on the local components of the metric.
\end{remark}
Recall that there is a canonical morphism of sheaves of unital superalgebras associated that defines the manifold $M_{red} := \big(|M|, C^\infty_{|M|}(-)\big)$; notationally we set $\epsilon_{-} :\cO_M(-) \rightarrow C^\infty_ {|M|}(-)$. A lift of $\bar{X}, \bar{Y} \in \Vect(M_{red})$ is defined as  any even $X,Y \in \Vect(M)$ such that $\bar{X} = X \circ \epsilon_{|M|}$ and $\bar{Y} = Y \circ \epsilon_{|M|}$. Such vector fields $X$ and $Y$ can always be found using an atlas of $M$.  Then we define a \emph{reduced metric} on $M_{red}$ as
$$\langle \bar{X} | \bar{Y}\rangle_{M_{red}}:=  \epsilon_{|M|}\big(  \langle X|Y \rangle \big)\,,$$
which, a priori, may be a degenerate. Using coordinates on $M_{red}$ induced by Shander coordinates $(x^a, t, \tau)$, the metric $g_{red}$ is of the form
$$g_{red} = \rmd x^a \otimes \rmd x^b \, g_{ba}(x,t)  + 2\, \rmd x^a \otimes \rmd t \,  g_{ta}(x,t)  + \rmd t \otimes \rmd t \,  g_{tt}(x,t)\,.$$
\begin{example}\label{exm:dim2NDgen}
Consider $\R^{2|1}$ equipped with global coordinates $(x, t, \tau)$, and the degenerate metric
$$g = \rmd x \otimes \rmd x - 2 \, \rmd t \otimes \rmd \tau \, \tau + \rmd t \otimes \rmd t\,.$$
The reader can quickly check that we do indeed have a super-null Riemannian manifold, i.e., the kernel of $g$ is generated by $Q$. The reduced metric is $g_{red} = \rmd x \otimes \rmd x + \rmd t \otimes \rmd t$, which is non-degenerate.
\end{example}
\begin{proposition}\label{prop:RedMisRie}
Let $(M, g, Q, P)$ be a super-null Riemannian manifold. Then the reduced manifold $M_{red}$ is a pseudo-Riemannian manifold.
\end{proposition}
\begin{proof}
As the question of degeneracy can be addressed locally, we will employ Shander coordinates $(x^a,t,\tau)$,  so that $Q=\partial_\tau+\tau\partial_t$ and $P=\partial_t$.
By Proposition \ref{prop:LocalFormg}, the reduced metric has the form
$$g_{red} = \rmd x^a \otimes \rmd x^b \, g_{ba}(x,t)  + 2\, \rmd x^a \otimes \rmd t \,  g_{ta}(x,t)  + \rmd t \otimes \rmd t \,  g_{tt}(x,t)\,.$$
\noindent \ul{$n= 1$}:  Degenerate metrics  are locally of the form
$$g  = - 2\, \rmd t \otimes \rmd \tau \, \tau g_{tt}(t) + \rmd t \otimes \rmd t \, g_{tt}(t)\,,$$
where $g_{tt}$ is a smooth function $t$. To have a constant rank kernel, it must be the case that $g_{tt}$ is nowhere vanishing.  The reduced metric is $g_{red} = \rmd t \otimes \rmd t \,  g_{tt}(t)$, which is non-degenerate. \par
\noindent \ul{$n\geq 2$}: As a block matrix we have
$$g_{red}=\begin{pmatrix}g_{ab} & g_{a t}\\ g_{t b} & g_{tt}\end{pmatrix}\,,$$
where the entries are ordinary smooth functions in $x^a$ and $t$. Whenever $g_{ab}$ is invertible, we may factor the determinant by the Schur complement:
$$ \det (g_{red})=\det(g_{ab})\cdot S\,,\qquad  S:=g_{tt}-g_{t a},(g_{ab})^{-1}g_{b t}\,.$$
Note that $S$ is a scalar, and so $\det(S) = S$, meaning the above is just the block matrix expression for the determinant. \par
\medskip
\noindent \textbf{Claim 1:} for $n \geq 2$, $\det(g_{ab})\neq 0$.
We will prove this via a contradiction. Let us assume that there exists a $U = U^a \partial_a$ such that $U^ag_{ab}=0$. Then any vector field $X = X^a \partial_a + X^t\partial_t + X^\tau \partial_\tau$ we have
$$\langle U | X\rangle = U^a X^b \,g_{ba} + U^a X^t g_{ta}\,.$$
Next, consider a new vector field $U' = U +  \lambda, P$, where $\lambda$ is an even function. Then
$$\langle U' | X\rangle = U^a X^b \,g_{ba} + U^a X^t g_{ta} + \lambda \, \big(X^a g_{at} + X^tg_{tt} \big)\,.$$
For $\langle U' | X\rangle$ to vanish, we require
$$U^ag_{at} + \lambda \, g_{tt}=0\,, \qquad \lambda \, g_{at} =0\,.$$\
\noindent \textit{Case 1:} consider $g_{at}=0$, then we have $U^ag_{at} =0$ and thus  $\langle U | X\rangle =0$ for all $X$. This violates the condition  $\ker(g)$ is of rank $0|1$.\\
\noindent \textit{Case 2:} consider  $g_{at} \neq 0$ and $g_{tt} \neq 0$, then $\lambda =0$ and this forces $U^ag_{at} =0$ and thus  $\langle U | X\rangle =0$ for all $X$. This violates the condition  $\ker(g)$ is of rank $0|1$.\par
\medskip
Thus, there is no such $U$ and thus $g_{ab}$ is invertible.\par
\bigskip
\noindent \textbf{Claim 2:} for $n \geq 2$, The Schur scalar $S$ is nonzero. Claim 1 establishes that $g_{ab}$ is invertible, and we will denote the inverse as $g^{ab}$, as standard. The  Schur scalar  is $S = g_{tt}- g_{ta}g^{ab}g_{bt}$. If $S =0$, then there exists a vector field $V \in \ker(g_{red})$. Every such even vector field has a canonical lift, also denoted here by $V$, as it does not depend on the odd coordinate. Thus, $\epsilon(\langle V | X\rangle) =0$ for all lifts $X$, and so $\langle V | X\rangle =0$. Thus, $V$ is in the kernel of $g$. However, this violates the condition on the kernel’s rank. This implies it must be the case that $S \neq 0$. \par
\bigskip
Thus, for $n\geq 2$, $\det(g_ {red})\neq 0$ and so $g_{red}$ defines a pseudo-Riemannian structure on $M_{red}$.\end{proof}
\begin{example}
Let $(M_0, g_0)$ be a (pseudo-)Riemannian manifold with $\dim M_0 \geq 2$. Then given a nowhere vanishing function $f\in C^\infty(M_0)$ we have a warped product of non-degenerate metrics on $M_{red} := M_0 \times \R$ given by
$$g_{red} := g_0 \oplus g_\R \, f^2\,, $$
where $g_\R$ is the constant metric on $\R$. Then $M = M_0 \times \R^{1|1}$ is a super-null Riemannian manifold with the degenerate metric in Shander coordinates being of the form
$$g = \rmd x^a \otimes \rmd x^b g_{ba}(x) - 2 \, \rmd t \otimes \rmd \tau \, \tau\, f^2(x) + \rmd t \otimes \rmd t \, f^2(x)\,.$$
\end{example}
%
%%%%%%%%%%%%%%%%%%%%%%%
%
\subsection{The Super-Null Lie Algebra}\label{subsec:SupCarLieAlg}
The super-null Lie algebra was defined earlier as the infinitesimal automorphisms of a super-null Riemannian manifold (see Definition \ref{def:SUSYCarBun}). The Lie superalgebra $\mathfrak{snull}(M, g, Q, P)$ is defined as the Lie superalgebra of vector fields $X \in \Vect(M)$ that satisfy $\mathcal{L}_X g =0$ and $\mathcal{L}_X Q =0$. In this subsection, we examine the structure of this Lie superalgebra and establish that it is finite dimensional.
\begin{proposition}\label{prop:DNotKilling}
Let $(M, g, Q, P)$ be a super-null Riemannian manifold. Then $Q$ cannot be a Killing vector field, i.e., $\mathcal{L}_Q g \neq 0$.  
\end{proposition}
\begin{proof}
Using \eqref{eqnm:LieDerMet} we observe that for an arbitrary $X \in\Vect(M)$ and using $[Q,Q] = 2 \,P$
$$(\mathcal{L}_Q g)(Q,X) = Q\langle Q| X \rangle - \langle [Q,Q]| X \rangle +  \langle Q| [Q,X] \rangle  =  - 2 \, \langle P |X \rangle\,.$$
Thus, as $P$ is not in the kernel of $g$ we cannot have $(\mathcal{L}_Q g)(Q,X) = 0$ for all $X$, and so $\mathcal{L}_Q g\neq 0$.
\end{proof}
\begin{corollary}
Let $(M, g, Q, P)$ be a super-null Riemannian manifold. The odd vector field $Q$ is not an element of $\mathfrak{snull}_1(M, g, Q, P)$.
\end{corollary}
\begin{remark}
Proposition \ref{prop:DNotKilling} should be contrasted with the situation in Carrollian geometry, where the fundamental vector field can be Killing for the degenerate metric. Riemannian supermanifolds with a Killing homological vector field were the subject of \cite{Bruce:2020}. 
\end{remark}
We can interpret Proposition \ref{prop:DNotKilling} as saying that if $Q$ were killing, then $\ker(g)$ must include $P$; and so the rank of the kernel must be (at least) $1|1$. However, this directly violates the requirements of the definition of a super-null Riemannian manifold. It is important to note that $P$ may be Killing, and this does not destroy the super-null Riemannian manifold structure. If $P$ is a Killing vector field, then the local components of the degenerate metric in Shander coordinates are independent of the even coordinate $t$.
\begin{lemma}\label{lem:supCarrLieAlg}
Let $(M, g, Q, P)$ be a super-null Riemannian manifold. In Shander coordinates $(x^a, t, \tau)$, elements of $\Vect(M)$ that (super)commute with $Q = \partial_\tau + \tau \, \partial_t$ are of the form\\
\begin{tabular}{ll}
$X = X^a(x)\partial_a + X^t(x)\partial_t$ & $\in \Vect_0(M)$\,,\\
$Y = f(x)\, D$ & $\in \Vect_1(M)$\,,
\end{tabular}\\
where $D= \partial_\tau - \tau \, \partial_t$ is the supersymmetic covariant derivative.
\end{lemma}
\begin{proof}
For even vector fields, using Shander coordinates the general from is $X = X^a(x,t)\partial_a + X^t(x,t)\partial_t + \tau \, X^\tau(x,t)\partial_\tau$. We then observe that
\begin{align*}
[X,Q] &= [X^a \partial_a, \partial_\tau] + [X^a \partial_a, \tau  \, \partial_t]+ [X^t \partial_t, \partial_\tau] + [X^t \partial_t, \tau  \, \partial_t] + [\tau \,X^\tau \partial_\tau, \partial_\tau] + [\tau \,X^\tau \partial_\tau, \tau  \, \partial_t]\\
&= - \tau \, \partial_t X^a\partial_a - \tau \, \partial_t X^t\partial_t - X^\tau \partial_\tau + \tau X^\tau \partial_t\,.
\end{align*}
For the above to be zero, we have $X^\tau =0$, $\partial_t X^a =0$, and $\partial_t X^t =0$.\\
\medskip
For odd vector fields, the general form in Shander coordinates is $Y = \tau\, Y^a(x,t)\partial_a + \tau \, Y^t(x,t)\partial_t + Y^\tau(x,t)\partial_\tau$. We then observe that 
\begin{align*}
[Y,Q] &= [\tau \,Y^a \partial_a, \partial_\tau] + [\tau \, Y^a \partial_a, \tau  \, \partial_t]+ [\tau\, Y^t \partial_t, \partial_\tau] + [\tau Y^t \partial_t, \tau  \, \partial_t] + [ Y^\tau \partial_\tau, \partial_\tau] + [Y^\tau \partial_\tau, \tau  \, \partial_t]\\
&= Y^a\partial_a + Y^t \partial_t + Y^\tau\partial_t + \tau \partial_t Y^\tau \partial_\tau\,.
\end{align*}
For the above to be zero, we have $Y^a =0$, $\partial_t Y^\tau =0$, and $Y^\tau = - Y^t$, and so setting $Y^t = - f(x)$ establishes the result.
\end{proof}
\begin{proposition}\label{prop:SCarrAlg}
Let $(M, g, Q, P)$ be a super-null Riemannian manifold and let  $\mathfrak{snull}(M, g, Q, P)$ be its super-null Lie algebra. 
\begin{enumerate}[itemsep=0.5em]
\item The Lie algebra $\mathfrak{snull}_0(M, g, Q, P)$ is isomorphic to the Lie algebra of Killing vector fields of $(M_{red}, g_{red})$ that commute with $P_{red}$.
\item If $P$ is Killing, then  $\mathfrak{snull}_1(M, g, Q, P) = \Span_{\R}\{D\}$, otherwise  $\mathfrak{snull}_1(M, g, Q, P) = \{\mathbf{0}\}$. 
\end{enumerate}
\end{proposition}
\begin{proof} We will use Lemma \ref{lem:supCarrLieAlg} and Shander coordinates.
\begin{enumerate}[itemsep=0.5em]
\item Note if $[X, Q] =0$, then the even vector field must locally be of the form $X = X^a(x)\partial_a + X^t(x)\partial_t$, and in particular there is no $\tau$ component. Moreover, $X$ is projectable to $M_{red}$, the associated vector field on $M_{red}$ we denote by $X_{red}$.   If $X$ is Killing for $g$, then 
$$(\mathcal{L}_X g)_{ab} =0\, , \qquad (\mathcal{L}_X g)_{ta} =0\, , \qquad (\mathcal{L}_X g)_{tt} =0\,, $$
which implies that $X_{red}$ is Killing for $g_{red}$ (see Proposition \ref{prop:RedMisRie}). Then using Proposition \ref{prop:LocalFormg}, the further conditions for $X$ to be Killing are
$$(\mathcal{L}_X g)_{\tau a} \propto \tau \, (\mathcal{L}_X g)_{ta} =0\, , \qquad (\mathcal{L}_X g)_{\tau t } \propto \tau \,(\mathcal{L}_X g)_{tt} =0\,, $$
which are automatically satisfied. Thus, $X$ is fully determined by $X_{red}$ being Killing for $g_{red}$ subject to $[P_{red}, X_{red}] =0$, i.e., Killing vector fields that are locally independent of $t$.
\item Note  that if $[Y, Q]=0$, then the non-zero components of the odd vector field are $Y^t = - \tau \, f(x)$ and $Y^\tau = f(x)$. We observe that to be Killing we require
\begin{align*}
(\mathcal{L}_Y g)_{ba} &= f \, \partial_t g_{ba} +  2(\partial_a f) g_{tb} +2(\partial_b f) g_{ta} =0\,,\\
(\mathcal{L}_Y g)_{ta} & = f \, \partial_t g_{ta} +  2(\partial_a f) g_{tt} = 0\,,\\
(\mathcal{L}_Y g)_{tt} &= - \tau f(x)\, \partial_t g_{tt} =0\,. 
\end{align*}
\begin{itemize}
\item If $P$ is Killing, then $\partial_t g_{ba} =0$, $\partial_t g_{ta}=0$, and $\partial_t g_{tt}=0$. Then for $Y$ to be Killing we require $\partial_a f =0$, and so $f = \mathrm{const}$.  Thus, $\mathfrak{snull}_1(M, g, Q, P) = \Span_{\R}\{D\}$.
\item If $P$ is not  Killing, then $\partial_t g_{ba} \neq 0$, $\partial_t g_{ta}\neq0$, and $\partial_t g_{tt}\neq 0$. Then for $Y$ to be Killing the condition $f =0$ is forced. Thus, $\mathfrak{snull}_1(M, g, Q, P) = \{\mathbf{0}\}$.
\end{itemize}
\end{enumerate}
\end{proof}
As $(M_{red}, g_{red})$ is a pseudo-Riemannian manifold (see Proposition \ref{prop:RedMisRie}), the dimension of the Lie algebra of $g_{red}$ is bounded by $n(n+1)/2$, we have the following corollary of Proposition \ref{prop:SCarrAlg}.
\begin{corollary}\label{cor:SCarrAlfFinite}
Let $(M, g, Q, P)$ be a super-null Riemannian manifold, then the Lie superalgebra $\mathfrak{snull}(M, g, Q, P)$ is finite-dimensional.
\end{corollary}
\begin{remark}
The finite dimensionality of $\mathfrak{snull}(M, g, Q, P)$ is to be contrasted with infinitesimal automorphisms of weak Carrollian manifolds where the Lie algebra is infinite-dimensional (see Duval et al. \cite{Duval:2014}).
\end{remark}
Given the role of $P$ in determining the nature of the infinitesimal automorphisms, we make the following definition.
\begin{definition}
A super-null Riemannian manifold $(M,g, Q,P)$ such that $P$ is Killing, i.e., $\mathcal{L}_P g =0$, is referred to as a \emph{static super-null Riemannian manifold}. 
\end{definition}
Note that for any static super-null Riemannian manifold $P \in\mathfrak{snull}_0(M, g, Q, P)$ as $\mathcal{L}_P Q = [P,Q]=0$. Thus, the super-null Lie algebra of a static super-null Riemannian manifold consists of at least $P$ and $D$. As standard, we have a representation of the $N=1$, $d=1$ supertranslation algebra given by $[D,D] = - 2P$ and $[P,Q] =0$. Thus, we have established the following.
\begin{proposition}
Let $(M,g, Q,P)$ be a static super-null Riemannian manifold, then the $N=1$, $d=1$ supertranslation algebra is a Lie subsuperalgebra of $\mathfrak{snull}_(M, g, Q, P)$ realised by the vector fields $P$ and $D$.
\end{proposition}
\begin{example}\label{exa:SUSYLie}
Continuing  the \hyperref[exm:SUSY]{Motivating Example}, the  supermanifold $\R^{4|1}$ can canonically be equipped with the degenerate metric by employing global Shander coordinates $(x^a, t , \tau)$ and defining
$$g   = \rmd x^a \otimes \rmd x^b \, \delta_{ba}  + 2 \, \rmd t \otimes \rmd \tau \, \tau - \rmd t \otimes \rmd t\,.$$
The reduced manifold is $\R^n$, and the pseudo-Riemannian metric is 
$$g_{red} =  \rmd x^a \otimes \rmd x^b \, \delta_{ba} - \rmd t \otimes \rmd t\,,$$
i.e.,  we have the usual Minkowski spacetime of signature $(3,1)$.  
Clearly, $P$ is Killing and we have a static super-null Riemannian manifold, so the super-null Lie algebra contains $D$. The isometries of $g_{red}$ are given by the Poincar\'e Lie algebra, i.e., $\mathfrak{iso}(g_{red})\simeq \mathfrak{so}(3,1)\ltimes \R^4$. However, we require the Lie subalgebra generated by $P_{red} = \partial_t$ and the other generators of the Poincar\'e Lie algebra that commute with $P_{red}$. The remaining transformations are spatial rotations, spatial translations, and temporal translations. Importantly, there are no boosts. Thus, 
$$\mathfrak{scarr}(\R^{4|1}) \simeq\big(\mathfrak{e}(3)\oplus \mathfrak{u}(1) \big)\oplus_{\textrm Ext} \R^{0|1}\,, $$
where $\oplus_{\textrm Ext}$ denotes the odd, non-central extension of the even algebra defined by $[D,D] = -2\,P$. This Lie superalgebra is not the super-Poincar\'e algebra.
\end{example}
\begin{example}
Let $(M_0, g_0)$ be a (pseudo-)Riemannian manifold with $\dim M_0 \geq 2$, and let $\mathfrak{iso}(g_0)$ be its isometry Lie algebra. We then equip $M_{red} := M_0 \times \R$ with the product metric $g_{red} := g_0 \oplus g_\R$, where $g_\R$ is the standard constant metric on $\R$.  An established result is that $\mathfrak{iso}(g_0\oplus g_\R) \simeq \mathfrak{iso}(g_0)\oplus \mathfrak{iso}(g_\R)$. Recall that $\mathfrak{iso}(g_R)$ consists of just translations and thus is identified with $\mathfrak{u}(1)$.  Then $M := M_0 \times \R^{1|1}$ is a super-null Riemannian manifold whose degenerate metric written in Shander coordinates is 
$$g = \rmd x^a \otimes \rmd x^b g_{ba}(x)- 2 \, \rmd t \otimes \rmd \tau \, \tau + \rmd t \otimes \rmd t\,,$$
where $x^a$ form a coordinate system on $M_0$. As $P = \partial_t$  is Killing, we have
$$\mathfrak{snull}(M_0\times \R^{1|1}) \simeq\big(\mathfrak{iso}(M_0)\oplus \mathfrak{u}(1) \big)\oplus_{\textrm Ext} \R^{0|1}\,, $$
where $\oplus_{\textrm Ext}$ denotes the odd, non-central extension of the even algebra defined by $[D,D] = -2\,P$.
\end{example}
%
%%%%%%%%%%%%%%%%%%%%%%%
%
\subsection{Connections on Super-Null Riemannian Manifolds}\label{subsec:ConnSUSYBun}
Affine connections on any (real and finite dimensional) supermanifold always exist. The question is one of compatibility conditions and how these affect the existence of these connections.
\begin{definition}
Let $(M, g, Q, P)$ be a super-null Riemannian manifold. An affine connection on the supermanifold $M$ is said to be
\begin{enumerate}[itemsep=0.5em]
\item \emph{supersymmetry compatible} if $Q$ is parallel, i.e., $\nabla_X Q =0$ for all $X \in \Vect(M)$;
\item \emph{metric compatible} if $X\langle Y|Z \rangle  = \langle \nabla_X Y|Z \rangle  + (-1)^{\widetilde{X}\, \widetilde{Y}} ~ \langle Y|\nabla_X Z \rangle$, for all $X,Y,Z \in \Vect(M)$;
\item \emph{compatible} if it is both supersymmetry compatible and metric compatible.
\end{enumerate}
\end{definition} 
\begin{remark}
As $M_{red}$ is a pseudo-Riemannian manifold (see Proposition \ref{prop:RedMisRie}), the reduced manifold canonically comes with the Levi-Civita connection. However, this connection does not by itself define an affine connection on $M$.
\end{remark}
\subsubsection*{Supersymmetry Compatible Connections}
We proceed to describe the fundamental properties of supersymmetry compatible connections and establish their existence.
\begin{proposition}
Let $(M, g, Q, P)$ be a super-null Riemannian manifold equipped with an affine connection $\nabla$. If the affine connection is supersymmetry compatible, then $\nabla$ cannot be torsionless.  
\end{proposition}
\begin{proof}
From the definition of the torsion tensor (see Definition \ref{def:SymCon}), we have 
$$T_\nabla(Q,Q) = \nabla_Q Q + \nabla_Q Q - [Q,Q] = - 2 P\,.$$
Thus, the torsion tensor does not vanish as $P$ is non-zero.
\end{proof}
In interpreting this result, the condition $\nabla_X Q =0$ means that $Q$ ``gives an odd straight direction''. However, the non-integrability of the kernel, thought of as a distribution, forces any supersymmetry compatible connection to carry torsion.\par 
Given any pair of vector fields $X,Y \in \Vect(M)$ and an affine connection we have the $C^\infty(M)$-linear map defined by the curvature, i.e.,  
$$R_\nabla(X,Y)  : \Vect(M) \longrightarrow \Vect(M)\,.$$
\begin{proposition}\label{prop:SUSYcondFlat}
Let $(M, g, Q, P)$ be a super-null Riemannian manifold equipped with an affine connection $\nabla$. If the affine connection is supersymmetry compatible, then $Q \in \ker\big( R_\nabla(X,Y)\big)$ for all pairs of vector fields $X,Y \in \Vect(M)$.
\end{proposition}
\begin{proof}
From Definition \ref{def:RieCurTensor}, we have
$$ R_\nabla(X,Y)Q =  \nabla_X (\nabla_Y Q) - (-1)^{\widetilde{X}\, \widetilde{Y}}\, \nabla_Y (\nabla_X Q) - \nabla_{[X,Y]}Q\,.$$  
If the affine connection is supersymmetry compatible, i.e., $\nabla_X Q =0$, then $R_\nabla(X,Y)Q = 0$, for all $X,Y \in \Vect(M)$.
\end{proof}
The interpretation of Proposition \ref{prop:SUSYcondFlat} is that the affine connection is flat in the ``direction'' of $Q$ in the fibres. While this partial flatness condition is restrictive, we have the following theorem.
\begin{theorem}\label{thm:ExSUSYCon}
Supersymmetry compatible connections exist on  any super-null Riemannian manifold. 
\end{theorem}
\begin{proof}
Affine connections always exist on a supermanifold, and so we select one $\nabla^0$.  We 
then define a new connection given by $\nabla_X Y := \nabla^0_X Y + \Gamma(X,Y)$. Here $\Gamma$ is an even $(1,2)$-tensor.  Imposing the supersymmetry compatible condition, $\nabla_X Q := \nabla^0_X Q + \Gamma(X,Q) =0$, implies  $\Gamma(X,Q) = - \nabla^0_X Q $. As $Q$ is non-singular, the dual one-form  $\omega$ exists, that is, there is a one-form on $M$ such that $\omega(Q) =1$.  We can then define $\Gamma(X,Y) := \big(\nabla^0_X Q\big)\, \omega(Y)$. Thus, the affine connection 
$$\nabla_X Y :=   \nabla^0_X Y -  \big(\nabla^0_X Q\big)\, \omega(Y)\,,$$
exists and is a supersymmetry compatible connection.
\end{proof}
\begin{remark}
Note that the connection defined in the proof of Theorem \ref{thm:ExSUSYCon} is not unique and one can define $\nabla'_X Y :=   \nabla^0_X Y -  \big(\nabla^0_X Q\big)\, \omega(Y) + K(X,Y)$, where  $K$ is an even $(1,2)$-tensor such that $K(X,Q) \in \ker(g)$.
\end{remark}
\subsubsection*{Metric Compatible Connections}
We now repeat the analysis for metric compatible affine connection, accumulating with establishing their existence.
\begin{proposition}
Let $(M, g, Q, P)$ be a super-null Riemannian manifold equipped with a metric compatible connection. For every $X \in \Vect(M)$ there exists a function $f_X \in C^\infty(M)$ ($\widetilde{f_X} = \widetilde{X}$), such that $\nabla_X Q = f_X \, Q$.
\end{proposition}
\begin{proof}
From the metric compatibility condition $X\langle Q|Y \rangle = \langle \nabla_X Q|Y \rangle =0$, for all $X,Y \in \Vect(M)$. As $\ker(g) = \Span \{Q\}$, there must exist a function $f_X \in C^\infty(M)$ with $\widetilde{f_X} = \widetilde{X}$ such that $\nabla_X Q = f_X \, Q$.
\end{proof}
\begin{proposition}
Let $(M, g, Q, P)$ be a super-null Riemannian manifold equipped with an affine connection $\nabla$. If the affine connection is metric compatible, then $\nabla$ cannot be torsionless.  
\end{proposition}
\begin{proof}
From the definition of the torsion tensor (see Definition \ref{def:SymCon}), we have 
$$T_\nabla(Q,Q) = \nabla_Q Q + \nabla_Q Q - [Q,Q] = 2 f_Q \, Q - 2 P\,.$$
For the torsion to vanish, we require either
\begin{enumerate}
\item   $f_Q =0$ and $P =0$, but this is impossible as $P$ is non-zero; or
\item   $f_Q \,Q = P$, but this is impossible as $Q$ is odd and $P$ is even, meaning that they are linearly independent.
\end{enumerate}
Thus, the torsion tensor does not vanish.
\end{proof}
Similarly to the case of supersymmetric affine connections, the fact that the kernel of the metric is non-integrable forces metric compatible affine connections to carry torsion.
\begin{proposition}\label{prop:MetcondDirc}
Let $(M, g, Q, P)$ be a super-null Riemannian manifold equipped with an affine connection. If the affine connection is metric compatible, then $Q$ is a generalised eigenvector of $R_\nabla(X,Y)$ for all pairs of vector fields $X,Y \in \Vect(M)$.
\end{proposition}
\begin{proof}
From Definition \ref{def:RieCurTensor}, we have
$$ R_\nabla(X,Y)Q =  \nabla_X (\nabla_Y Q) - (-1)^{\widetilde{X}\, \widetilde{Y}}\, \nabla_Y (\nabla_X Q) - \nabla_{[X,Y]}Q\,.$$  
If the affine connection is metric compatible, i.e., $\nabla_X Q =f_X \, Q$, then 
\begin{align*}
R_\nabla(X,Y)Q & = \nabla_X (f_Y Q) - (-1)^{\widetilde{X}\, \widetilde{Y}}\, \nabla_Y (f_X Q) - f_{[X,Y]}Q \\
& = X(f_Y)Q  + (-1)^{\widetilde{X}\, \widetilde{Y}}\,f_Y f_X Q   - (-1)^{\widetilde{X}\, \widetilde{Y}}\,  Y(f_X) Q - f_X f_Y Q - f_{[X,Y]}Q\\
& = \big ( X(f_Y)  - (-1)^{\widetilde{X}\, \widetilde{Y}}\,  Y(f_X)  - f_{[X,Y]}\big )\, Q\,,
\end{align*}
for all $X,Y \in \Vect(M)$. Thus, $R(X,Y) Q = \hat{f}_{X,Y}\, Q$, where $\hat{f}_{X,Y} =X(f_Y)  - (-1)^{\widetilde{X}\, \widetilde{Y}}\,  Y(f_X) - f_{[X,Y]} \in C^\infty(M)$ is the generalised eigenvalue.
\end{proof}
Due to the degeneracy of the metric, the corresponding Koszul formula 
 \begin{align*}
2 \langle \nabla_X Y | Z  \rangle& = X\langle Y | Z   \rangle  + \langle [X,Y]| Z \rangle  + \langle T_\nabla(X,Y)|Z \rangle \\ \nonumber
 & +(-1)^{\widetilde{X}\, (\widetilde{Y} + \widetilde{Z}) }~ \big( Y \langle Z|X\rangle - \langle [Y,Z]| X \rangle + \langle T_\nabla(Y,Z)|X \rangle \big)\\ \nonumber
 &-(-1)^{\widetilde{Z}\, ( \widetilde{X} + \widetilde{Y} )}~ \big( Z \langle X|Y\rangle - \langle [Z,X]| Y \rangle + \langle T_\nabla(Z,X)|Y \rangle\big),
\end{align*}
does not fully determine a metric compatible connection. While $\langle \nabla_X Y | Z  \rangle $ is fully determined for a given metric compatible connection (assuming at least one exists), $\nabla_X Y$ is only determined up to a vector in $\ker(g)$. Thus, we cannot expect a direct analogue of the fundamental theorem of Riemannian supergeometry for super-null Riemannian manifolds. Nonetheless, we have the following theorem.
\begin{theorem}\label{thm:ExMetricCon}
Metric compatible affine connections exist on any super-null Riemannian manifold. 
\end{theorem}
\begin{proof}
Affine connections always exist on a supermanifold, and so we select one $\nabla^0$. We make no assumption about the connection, such as being torsion-free or metric compatible. Thus, the non-metricity is measured by the following tensor
$$\big(\nabla^0_X g\big )(Y,Z) = X \langle Y| Z \rangle - \langle \nabla^0_X Y | Z\rangle - (-1)^{\widetilde{X}\widetilde{Y}}\, \langle  Y | \nabla^0_X Z\rangle\,.$$
We now define a new connection given by $\nabla_X Y := \nabla^0_X Y + \Gamma(X,Y)$, here $\Gamma$ is an even $(1,2)$-tensor.  Imposing the metric compatibility condition on $\nabla$ forces an algebraic constraint on $\Gamma$ which we will examine. Specifically,
$$X\langle Y|Z\rangle = \langle \nabla^0_X Y |Z  \rangle + \langle \Gamma(X,Y) |Z\rangle +(-1)^{\widetilde{X} \widetilde{Y}}\, \langle Y| \nabla^0_X Z \rangle + (-1)^{\widetilde{X} \widetilde{Y}}\, \langle Y| \Gamma(X,Z) \rangle\,.$$ 
Using the non-metricity of $\nabla^0$, the algebraic condition 
 $$\big(\nabla^0_X g \big)(Y,Z) = \langle \Gamma(X,Y) |Z\rangle + (-1)^{\widetilde{X} \widetilde{Y}}\, \langle Y| \Gamma(X,Z) \rangle\,,$$
is deduced. The degeneracy of the pseudo-inner product does not fully constrain the components of the tensor $\Gamma$; there is freedom in choosing components of $\Gamma$ that lie in the kernel of $g$. The number of components of $\Gamma$ is greater than the number of independent equations. The system of linear equations is underdetermined, and so a solution can always be found (there are, in fact, an infinite number of solutions). Thus, a metric compatible affine connection can always be constructed from an arbitrary affine connection, and so the theorem is established.
 \end{proof}
\begin{remark}
We stress that the connection built in the proof of Theorem \ref{thm:ExMetricCon} is far from unique. As the components of $\Gamma$ that lie the kernel of the metric are not constrained, one can always modify the metric compatible connection as $\nabla'_X Y =  \nabla^0_X Y + \Gamma(X,Y)+ K(X,Y)$, where $K(X,Y)$ is an even $(1,2)$-tensor such that $K(X,Y) \in \ker(g)$ for all $X,Y  \in \Vect(M)$. This freedom will play a vital role in constructing compatible connections. 
\end{remark}
\subsubsection*{Compatible Connections}
Combining the results of this subsection, we are led to the main theorem of this paper. We observe that for compatible connections, we must have $f_X =0$ for all vector fields $X \in \Vect(M)$; while this is consistent, it is not immediate that such affine connections can be found.  
\begin{main} \label{mtrm:ComConExist}
Compatible affine connections exist on any super-null Riemannian manifold.
\end{main}
\begin{proof}
Starting from an arbitrary affine connection $\nabla^0$, the proof of Theorem \ref{thm:ExSUSYCon} show that we have a supersymmetric compatible connection given by
$$\nabla^1_X Y :=   \nabla^0_X Y -  \big(\nabla^0_X Q\big)\, \omega(Y)\,.$$
The proof of Theorem \ref{thm:ExMetricCon} allows us to  amend $\nabla^1$ to obtain a metric compatible connection given by
$$\nabla_X Y :=  \nabla^1_X Y + \Gamma(X,Y)\,.$$
We now need to impose the supersymmetric compatibility condition to further constrain $\Gamma$.  Directly, $\nabla_X Q =  \nabla^1_X Q + \Gamma(X,Q) =0$, which implies 
$$\Gamma(X,Q)=0\,,$$
for all $X \in \Vect(M)$. Next, we need to argue that this extra constraint can be satisfied while not destroying the metric compatibility constraint.  The non-metricity of $\nabla^1$,
$$(\nabla^1_X g)(Y,Z) = X \langle Y |Z \rangle -  \langle \nabla^1_X Y | Z\rangle - (-1)^{\widetilde{X} \widetilde{Y}}\, \langle Y | \nabla^1_X Z \rangle\,,$$ 
together with the algebraic condition on $\Gamma$,
$$(\nabla^1_X g)(Y,Z) =  \langle\Gamma(X,Y)| Z \rangle + (-1)^{\widetilde{X} \widetilde{Y}}\, \langle Y | \Gamma(X,Z)\rangle\,,$$ 
implies the following;
$$(\nabla^1_X g)(Q,Z) = - \langle \nabla^1_X Q | Z\rangle  =  \langle\Gamma(X,Q)| Z \rangle =0 \,,$$
as $\nabla^1_X Q =0$ by construction. Thus, $\Gamma(X,Q) \in \ker(g)$ for all $X \in \Vect(M)$. The proof of Theorem \ref{thm:ExMetricCon} shows that the components of $\Gamma$ that lie in the kernel of the metric are not constrained by the metric compatibility. Thus, we can choose $\Gamma(X,Q) =0$ and still have metric compatibility.  This establishes the result. 
\end{proof}
\begin{example}\label{exa:Connection}
Continuing Example \ref{exa:SUSYLie},  the  supermanifold $\R^{4|1}$ can canonically be equipped with the degenerate metric by employing global Shander coordinates $(x^a, t , \tau)$ and defining 
$$g   = \rmd x^a \otimes \rmd x^b \, \delta_{ba}  + 2 \, \rmd t \otimes \rmd \tau \, \tau - \rmd t \otimes \rmd t\,.$$
As we have a superdomain, we can chose $\nabla^0$ to be the trivial connection and globally set $\omega = \rmd \tau$ (understood as an odd one-from). Then
$$\nabla^1_X Y := X(Y)- X(\tau)\, \rmd \tau(Y)\partial_t\,,$$
defines a supersymmertic compatible connection. However, this connection is not metric compatible. A minimal choice of $\Gamma$ is $\Gamma(\partial_\tau,\partial_\tau) = 2 \partial_t$ and all other components are zero. Then 
$$\nabla_X Y := X(Y)- X(\tau)\, \rmd \tau(Y)\partial_t + \Gamma(X,Y)\,,$$
is a compatible affine connection. In this specific case, the non-vanishing component of the torsion is identified with $\Gamma(\partial_\tau,\partial_\tau)$ which is non-vanishing.
\end{example}
%
%%%%%%%%%%%%%%%%%%%%%%%
%
\section{Concluding Remarks}\label{sec:ConRem}
An odd analogue of a Carrollian manifold has been constructed and studied; we refer to such supermanifolds as super-null manifolds. We have shown that supersymmetry compatible and metric compatible connections exist on any super-null Riemannian manifold, and importantly, that compatible connections always exist, i.e., affine connections that are both supersymmetry compatible and metric compatible can always be constructed.  It is the freedom in defining a connection, thanks to the degeneracy of the metric, that allows these two conditions to be simultaneously satisfied. It was argued by Bekaert \& Morand \cite{Bekaert:2018} that only invariant Carrollian manifolds, i.e., Carrollian manifolds for which the fundamental/Carrollian vector field is Killing, can admit torsion-free compatible connections.  This is in stark contrast with super-null Riemannian manifolds, where the supersymmetry generator $Q$ cannot be Killing, and compatible connections must carry torsion. The two geometries are fundamentally different; the non-integrability of the kernel of the super-case is the root of these differences.\par 
The Lie superalgebra of infinitesimal automorphisms of a super-null Riemannian manifold, which we referred to as the super-null Lie algebra, has been studied. An interesting result is that this Lie superalgebra is finite-dimensional and tightly tied to the Lie algebra of Killing vector fields of the reduced metric. In the classical setting of weak Carrollian manifolds, the Lie algebra of infinitesimal automorphisms is infinite-dimensional. When the extra condition of preserving a compatible connection is imposed, the Lie algebra becomes finite-dimensional; see Duval et al. \cite{Duval:2014} and references therein for details. \par
While the main motivation for this work is rooted in mathematical curiosity, further explicit examples of super-null Riemannian manifolds are desirable and could expose applications thereof in physics. It may be possible to formulate a superspace version of non-holonomic supermechanics where the dynamics generated by $Q$ could depend on the extra even coordinates understood as ``external parameters''; which might be interpreted as the classical background fields or slowly varying time-dependent parameters of the physical system. This could provide a novel way to study supersymmetric mechanical systems with evolving external parameters. For example, there may be supersymmetric analogues of geometric phases, such as the Hannay angle (see \cite{Hannay:1985}), underlying these models. 
%
%%%%%%%%%%%%%%%%%%%%%%%
%
\section*{Acknowledgments}
The author extends his gratitude to Steven Duplij and Janusz Grabowski for their encouragement.    
%
%%%%%%%%%%%%%%%%%%%%%%%
%

\end{document}